\documentclass{amsart}

\usepackage{ae,aecompl} 
\usepackage{chngcntr}
\usepackage{graphicx}
\usepackage{amsmath}
\usepackage{amsthm}
\usepackage{amssymb}
\usepackage{amsfonts}
\usepackage{amsrefs}
\usepackage{qsymbols}
\usepackage{latexsym}
\usepackage{chngcntr}
\usepackage{cite}
\usepackage{paralist}

\newtheorem{theorem}{Theorem}[section]
\newtheorem{lemma}[theorem]{Lemma}
\newtheorem{proposition}[theorem]{Proposition}
\newtheorem{corollary}[theorem]{Corollary}

\theoremstyle{definition}
\newtheorem{definition}[theorem]{Definition}
\newtheorem{remark}[theorem]{Remark}

\numberwithin{equation}{section}

\newcommand{\IR}{\mathbb{R}}
\newcommand{\IC}{\mathbb{C}}
\newcommand{\IN}{\mathbb{N}}

\newcommand{\abs}[1]{\lvert#1\rvert}
\newcommand{\e}{\mathrm{e}}

\renewcommand{\d}{\mathrm{d}}
\renewcommand{\i}{\mathrm{i}}

\newcommand{\ind}{{\mathbf{1}}}
\newcommand{\Ell}{\mathrm{L}} 
\newcommand{\Ce}{\mathrm{C}}
\newcommand {\HT}{{\mathrm{H}}}

\DeclareMathOperator{\supp}{supp}

\DeclareMathOperator{\Real}{Re}

\DeclareMathOperator{\D}{D}

\title{Convergence of subdiagonal Pad\'{e} approximations of $C_{0}$-semigroups}

\author{Moritz Egert}
\address{Fachbereich Mathematik, Technische Universit\"at Darmstadt, Schlossgartenstr. 7, 64289 Darmstadt, Germany}
\email{egert@mathematik.tu-darmstadt.de}
\thanks{}

\author{Jan Rozendaal}
\address{EEMCS, Delft University of Technology, Mekelweg 4, 2628CD Delft, The Netherlands}
\email{J.Rozendaal-1@tudelft.nl}
\thanks{}


\date{\today}

\dedicatory{}


\begin{document}
\begin{abstract}
Let $(r_{n})_{n \in \IN}$ be the sequence of subdiagonal Pad\'{e} approximations of the exponential function. We prove that for $-A$ the generator of a uniformly bounded $C_{0}$-semigroup $T$ on a Banach space $X$, the sequence $(r_{n}(-tA))_{n \in\IN}$ converges strongly to $T(t)$ on $\D(A^{\alpha})$ for $\alpha>\frac{1}{2}$.  Local uniform convergence in $t$ and explicit convergence rates in $n$ are established. For specific classes of semigroups, such as bounded analytic or exponentially $\gamma$-stable ones, stronger estimates are proved. Finally, applications to the inversion of the vector-valued Laplace transform are given.
\end{abstract}

\maketitle

\section{Introduction}
\label{Introduction}

\noindent In \cite[Rem.\ 3]{Jara} the question was raised, whether there are complex numbers $\lambda_{n,m}$ and $b_{n,m}$ such that any uniformly bounded $C_{0}$-semigroup $(T(t))_{t \geq 0}$ with generator $-A$ on a Banach space $X$ can be approximated in the strong sense on the domain $\D(A)$ of $A$ by sums of the form
\begin{align*}
 \frac{b_{n,1}}{t}\bigg(\frac{\lambda_{n,1}}{t} +A\bigg)^{-1} + \dots + \frac{b_{n,m_n}}{t}\bigg(\frac{\lambda_{n,m_n}}{t}+A\bigg)^{-1} \quad \quad (t>0),
\end{align*}
as $n \to \infty$, locally uniformly in $t \geq 0$, see also \cite{NOS}. According to \cite{NOS} such a method is called \emph{rational approximation without scaling and squaring}, because of the absence of both the successive squaring of the resolvent and the scaling of the generator by $\frac{1}{n}$ that is common to other approximation methods for the choice $n=2^{k}$, $k\in\IN$, see e.g.\@~\cite{Brenner-Thomee}. Recently, the rational approximation method of semigroups above, and others, have been used to provide new powerful inversion formulas for the vector-valued Laplace transform \cite{NOS}, \cite{Jara}.

\medskip

Suppose $(r_n)_{n \in \IN}$ is a sequence of rational functions such that the degree of the numerator of $r_{n}$ is less than the degree of its denominator and such that each $r_n$ has pairwise distinct poles $\lambda_{n,m}$ which all lie in the open right halfplane $\IC_+$. Developing $r_n$ into partial fractions, there are complex numbers $b_{n,m}$ such that
\begin{align*}
 r_n(z) = \frac{b_{n,1}}{\lambda_{n,1}-z} + \dots + \frac{b_{n,m_n}}{\lambda_{n,m_n}-z} \quad \quad (z \in \IC \setminus \{\lambda_{n,1},\dots,\lambda_{n,m_n} \}).
\end{align*}
If $-A$ generates a uniformly bounded $C_{0}$-semigroup, the open right halfplane belongs to the resolvent set of $-tA$ for any $t\geq 0$. By the Hille-Phillips calculus,
\begin{align*}
 r_n(-tA) = \frac{b_{n,1}}{t}\bigg(\frac{\lambda_{n,1}}{t}+A\bigg)^{-1} + \dots + \frac{b_{n,m_n}}{t}\bigg(\frac{\lambda_{n,m_n}}{t}+A\bigg)^{-1}.
\end{align*}
Hence, the original problem is solved, provided one can prove the following: there is a sequence $(r_n)_{n \in \IN}$ of rational functions satisfying the properties above and such that, for any generator $-A$ of a uniformly bounded semigroup $(T(t))_{t \geq 0}$ on a Banach space $X$, the convergence
\begin{align}
\label{goal}
 r_n(-tA)x \stackrel{n \to \infty}{\longrightarrow} T(t)x
\end{align}
holds for all $x \in \D(A)$, locally uniformly in $t \geq 0$.

\medskip

For bounded generators $-A$ this was achieved recently in \cite{Lee-Diss}. Therein, the author takes $r_n$ to be the $n$-th subdiagonal Pad\'{e} approximation to the exponential function and then uses a refinement of an error estimate shown in \cite{NOS}. For this choice of $r_n$, numerical experiments in \cite{Lee-Diss} hint at \eqref{goal} to hold with rate $\mathcal{O}(\frac{1}{\sqrt{n}})$ for any (unbounded) generator $-A$ of a uniformly bounded $C_0$-semigroup. However, no mathematical proof for this was known so far.

\medskip

Our main result is that, for $-A$ the generator of a uniformly bounded $C_{0}$-semigroup $T = (T(t))_{t \geq 0}$ on a complex Banach space $X$, $(r_n)_{n \in \IN}$ the sequence of subdiagonal Pad\'{e} approximations to the exponential function and $\alpha> \frac{1}{2}$, the desired convergence \eqref{goal} holds with rate
  \begin{align*}
   \mathcal{O}(n^{-\alpha + \frac{1}{2}})
  \end{align*}
for $x\in \D(A^{\alpha})$, locally uniformly in $t$. In particular, the choice $\alpha=1$ yields the rate suggested by the numerical experiments in \cite{Lee-Diss}. We also mention that in this case the approximation method without scaling and squaring converges with the same rate as is known for the classical scaling and squaring methods due to Brenner and Thom\'{e}e \cite{Brenner-Thomee}.

Moreover, we establish improvements in the following cases:

\begin{enumerate}[(i)]
 \item The semigroup $T$ is bounded analytic. In this case \eqref{goal} holds on $\D(A^{\alpha})$ for arbitrary $\alpha>0$ with rate
  \begin{align*}
   \mathcal{O}(n^{-\alpha})
  \end{align*}
 locally uniformly in $t$.

 \item The semigroup $T$ is exponentially $\gamma$-stable (see Definition~\ref{gamma-bounded}). Then, for $\alpha> 0$ arbitrary, \eqref{goal} holds with rate
  \begin{align*}
   \bigcap \limits_{a < \alpha} \mathcal{O}(n^{-a})
  \end{align*}
 on $\D(A^{\alpha})$, locally uniformly in $t$. This holds in particular for any exponentially stable $C_{0}$-semigroup on a Hilbert space.

 \item The operator $A$ has a bounded holomorphic functional calculus (see Section~\ref{Hinfty calculus}). In this case \eqref{goal} holds for any $\alpha> 0$ with rate 
 \begin{align*}
 \mathcal{O}(n^{-\alpha})
 \end{align*}
 on $\D(A^{\alpha})$, locally uniformly in $t$. In addition, one has local uniform convergence in $t$ on the whole space $X$. This applies in particular if $T$ is (similar to) a contraction semigroup on a Hilbert space. 
\end{enumerate}

\medskip

The paper is organized as follows. Section~\ref{Notation and background} gives a summary of the concepts that will be used throughout. Section~\ref{Technical estimates} provides several technical lemmas that are crucial for the proof of our main results, which in turn are precisely formulated and proved in Section~\ref{Proofs}. The proof of the main result for uniformly bounded $C_{0}$-semigroups is based on the Hille-Phillips calculus and a method for estimating the error introduced by Brenner and Thom\'{e}e in \cite{Brenner-Thomee}. The result for analytic semigroups relies on the holomorphic functional calculus for sectorial operators. In the case of exponentially $\gamma$-stable semigroups we use recent results from \cite{Haase-Rozendaal}. Extensions of the main results to intermediate spaces such as Favard spaces can be found in Section~\ref{intermediate spaces}. Section~\ref{Laplace transform} provides applications of our results to the inversion of the vector-valued Laplace transform.

\section{Notation and background}
\label{Notation and background}

\noindent The set of nonnegative reals is $\IR_{+}:=[0,\infty)$ and the standard complex right halfplane is 
\begin{align*}
	\IC_{+}:=\left\{z\in\IC\left|\Real(z)>0\right.\right\}.
\end{align*}
Any Banach space $X$ under consideration is taken over the complex numbers. The space of bounded linear operators on $X$ is denoted by $\mathcal{L}(X)$.

\subsection{Semigroups}
\label{Semigroups}

A \emph{$C_{0}$-semigroup} $T=(T(t))_{t\geq 0}\subseteq \mathcal{L}(X)$ is a strongly continuous representation of $(\IR_{+},+)$ on a Banach space $X$. The semigroup has \emph{type} $(M,\omega)$, for $M\geq 1$ and $\omega\in\IR$, if $\|T(t)\|\leq Me^{\omega t}$ for all $t\geq 0$. For each $C_{0}$-semigroup $T$ the \emph{growth bound} 
\begin{align*}
	\omega_{0}(T):=\inf\left\{w\in\IR\mid \exists M\geq 1: T\textrm{ has type }(M,\omega)\right\}
\end{align*}
is an element of $[-\infty,\infty)$. A $C_{0}$-semigroup $T$ is called \emph{uniformly bounded} if it has type $(M,0)$, and \emph{exponentially stable} if $\omega_{0}(T)<0$.

The domain of an operator $A$ on $X$ is denoted by $\D(A)$. The \emph{spectrum} of $A$ is denoted by $\sigma(A)\subseteq \IC$ and the \emph{resolvent set} by $\rho(A):=\IC\setminus \sigma(A)$. For $\lambda\in \rho(A)$ the \emph{resolvent operator} of $A$ at $\lambda$ is $R(\lambda,A):=(\lambda-A)^{-1}\in\mathcal{L}(X)$.

For $\varphi \in (0,\pi)$,  denote by 
\begin{align*}
 \Sigma_\varphi := \{z \in \IC \setminus \{0\} \mid \abs{\arg z} < \varphi \}
\end{align*}
the open sector with vertex $0$ and opening angle $2\varphi$ symmetric around the positive real axis. An operator $A$ is \emph{sectorial} of \emph{angle} $\varphi\in(0,\pi)$ if its spectrum is contained in $\overline{\Sigma_{\varphi}}$ and
\begin{align*}
\sup\left\{\|\lambda R(\lambda,A)\|\mid\lambda\in\IC\setminus \overline{\Sigma_{\psi}}\right\}<\infty
\end{align*}
for each $\psi\in (\varphi,\pi)$. If $-A$ generates a uniformly bounded $C_{0}$-semigroup, then $A$ is sectorial of angle $\frac{\pi}{2}$. A $C_{0}$-semigroup $T$ with generator $-A$ is \emph{bounded analytic} if $A$ is sectorial of angle $\varphi\in(0,\frac{\pi}{2})$, which in turn implies that $T$ is uniformly bounded.

\subsection{Laplace and Fourier transform}
\label{L transform}

Denote by $\mathbf{M}(\IR_{+})$ the space of bound\-ed regular complex-valued Borel measures on $\IR_{+}$, which form a Banach algebra under convolution with the variation norm $\|\cdot\|_{\mathbf{M}(\IR_{+})}$. Any $g\in \Ell^{1}(\IR_{+})$ defines an element $\mu_{g}\in\mathbf{M}(\IR_{+})$ by $\mu_{g}(\d t):=g(t)\,\d t$, where $\d t$ denotes Lebesgue measure on $\IR_{+}$. In this manner $\Ell^{1}(\IR_{+})$ is isometrically embedded into $\mathbf{M}(\IR_{+})$, and a function $g$ will be identified with its associated measure $\mu_{g}$ when convenient.

Let $\HT^{\infty}\!(\IC_{+})$ be the space of bounded holomorphic functions on $\IC_{+}$. This is a unital Banach algebra under pointwise operations and the norm
\begin{align*}
\|f\|_{\HT^{\infty}\!(\IC_{+})}:=\sup_{z\in\IC_{+}}\abs{f(z)} \quad\quad (f\in \HT^{\infty}\!(\IC_{+})).
\end{align*}
For $\mu\in\mathbf{M}(\IR_{+})$ the \emph{Laplace-Stieltjes transform} of $\mu$ is given by
\begin{align*}
 \hat{\mu}(z) := \int_0^\infty \e^{-z t} \, \mu(\d t) \qquad (z \in \IC_{+}),
\end{align*}
and $\mu \mapsto \hat{\mu}$ defines a contractive algebra homomorphism from $\mathbf{M}(\IR_{+})$ into $\HT^{\infty}\!(\IC_+) \cap \Ce(\overline{\IC_{+}})$.

The \emph{Fourier transform} of $f \in \Ell^{1}(\IR)$ is defined as
\begin{align*}
	(\mathcal{F} f)(\xi) := \int_\IR \e^{-\i x \xi}f(x) \, \d x \quad \quad (\xi \in \IR).
\end{align*}
Then $\mathcal{F}: \Ell^{1}(\IR) \to \Ell^{\infty}\!(\IR)$ is a contraction, and Plancherel's theorem asserts that $\mathcal{F}$ yields an almost isometric isomorphism $\mathcal{F}:\Ell^{2}(\IR) \to \Ell^{2}(\IR)$ with
\begin{align*}
	\|\mathcal{F}f\|_{\Ell^{2}(\IR)} = \sqrt{2 \pi} \|f\|_{\Ell^{2}(\IR)} \quad \quad (f \in \Ell^{2}(\IR^d)).
\end{align*}

\subsection{Functional calculus}
\label{functional calculus}

Let $(T(t))_{t \geq 0}$ be a $C_{0}$-semigroup of type $(M,0)$ on a Banach space $X$, with generator $-A$. For $\mu\in \mathbf{M}(\IR_{+})$ define
\begin{align*}
 \hat{\mu}(A)x:= \int_0^\infty T(t)x \, \mu(\d t)\quad \quad (x\in X).
\end{align*}
Then $\hat{\mu}(A)$ is a bounded operator on $X$ satisfying
\begin{align*}
\|\hat{\mu}(A)\| \leq M \|\mu\|_{\mathbf{M}(\IR_{+})}.
\end{align*}
Note that $\hat{\mu}(A)$ is well-defined since the Laplace transform is injective \cite[Thm.\@~6.3]{Widder}. The mapping
\begin{align*}
 \widehat{\mathbf{M}(\IR_{+})} \to \mathcal{L}(X),\quad \hat{\mu} \mapsto \hat{\mu}(A),
\end{align*}
is an algebra homomorphism, called the \emph{Hille-Phillips calculus} for  $A$. For more details on this concept see Chapter XV and Sections III.3.6 and IV.4.16 of \cite{Hille-Phillips}.

The Hille-Phillips calculus can be extended to a larger class of functions on $\IC_{+}$ by \emph{regularization} \cite[Sect.~1.2]{Haase}. If $f$ is a holomorphic function on $\IC_{+}$ for which there exists an $e\in \widehat{\mathbf{M}(\IR_{+})}$ such that $ef\in \widehat{\mathbf{M}(\IR_{+})}$ and $e(A)$ is injective, define
\begin{align*}
f(A):=e(A)^{-1}(ef)(A).
\end{align*}
This yields an (in general) unbounded operator on $X$, and the definition is independent of the choice of the regularizer $e$. It is consistent with the previous definition of $f(A)$ for $f\in \widehat{\mathbf{M}(\IR_{+})}$. 

\medskip

If $A$ is a sectorial operator of angle $\varphi\in(0,\pi)$, the holomorphic functional calculus for $A$ is defined, according to \cite[Ch.~2]{Haase}, as follows. For $\psi\in (\varphi,\pi)$ set
\begin{align*}
\HT_0^\infty\!(\Sigma_\psi) := \left\{g: \Sigma_\psi \to \IC \, \text{hol.} \mid \exists \, C,s>0 \, \forall \, z \in \Sigma_\psi: \abs{g(z)} \leq  C \min\{\abs{z}^s, \abs{z}^{-s}\}\right\}
\end{align*}
and then define $f(A)\in \mathcal{L}(X)$ for given $f\in \HT_0^\infty\!(\Sigma_\psi)$ as
\[
f(A):=\frac{1}{2\pi i}\int_{\partial \Sigma_{\nu}}f(z)R(z,A)\,\d z,
\]
where $\nu \in(\varphi,\psi)$ and $\partial \Sigma_{\nu}$ is the boundary curve of the sector $\Sigma_{\nu}$, oriented such that $\sigma(A)$ is surrounded counterclockwise. This integral converges absolutely and is independent of the choice of $\nu$, by Cauchy's theorem. Furthermore, define 
\begin{align*}
g(A):=f(A)+c(1+A)^{-1}+d
\end{align*}
if $g$ is of the form $g(\cdot)=f(\cdot)+c(1+\cdot)^{-1}+d$ for $f\in\HT^{\infty}_{0}\!(\Sigma_{\psi})$ and $c,d\in \IC$. This definition is independent of the particular representation of $g$ and yields an algebra homomorphism
\begin{align*}
\HT_0^\infty\!(\Sigma_\psi)\oplus\langle (1+\cdot)^{-1}\rangle\oplus\langle\ind\rangle\rightarrow \mathcal{L}(X),\quad g\mapsto g(A), 
\end{align*}
the \emph{holomorphic functional calculus} for the sectorial operator $A$. 

One then extends by regularization as above. In particular, for any $\alpha \in \IC_+$ the function $z\mapsto z^{\alpha}$ is regularizable by $z\mapsto (1+z)^{-n}$, where $n>\Real(\alpha)$, and yields the \emph{fractional power} $A^{\alpha}$ of $A$ with domain $\D(A^{\alpha})$. One has $A^1 = A$ and $A^0 = 1$.

\medskip

If $-A$ generates a uniformly bounded $C_{0}$-semigroup, then the Hille-Phillips calculus extends the holomorphic functional calculus for angles $\psi\in(\frac{\pi}{2},\pi)$, see Lemma 3.3.1 and Proposition 3.3.2 in \cite{Haase}. In particular, for $\alpha \in \IC_+$ the fractional power $A^{\alpha}$ can be defined in the Hille-Phillips calculus yielding the same operator as in the holomorphic functional calculus.

\subsection{Subdiagonal Pad\'{e} approximations}
\label{Pade approximations}

Throughout, $r_n:= \frac{P_n}{Q_n}$ denotes the $n$-th \emph{subdiagonal Pad\'{e} approximation} to the exponential function,  i.e.\@ $P_n$ and $Q_n$ are the unique polynomials of degree $n$ and $n+1$, respectively, such that $P_n(0) = Q_n(0) =1$ and 
\begin{align*}
 \left|r_n(z) - \e^z\right| \leq C \abs{z}^{2n+2}
\end{align*}
for $z \in \IC$ in a neighborhood of $0$. Such polynomials $P_n$ and $Q_n$ exist and are unique, and they take the explicit form
\begin{align}
\label{explicit pade}
\begin{split}
P_n(z) &= \sum_{j=0}^{n} \frac{(2n+1-j)!n!}{(2n+1)!j!(n-j)!} z^j, \\
Q_n(z) &= \sum_{j=0}^{n+1} \frac{(2n+1-j)!(n+1)!}{(2n+1)!j!(n+1-j)!} (-z)^j,
\end{split}
\end{align}
see e.g.\@ \cite[Thm.~3.11]{hairer2004solving}. As has already been observed by Perron \cite[Sect.~75]{perron1913lehre}, the error-term $r_n(z) - \e^z$ can be represented as
\begin{align}
\label{Perron}
 r_n(z) - \e^z = \frac{(-1)^{n+2}}{Q_n(z)} \frac{1}{(2n+1)!} z^{2n+2} \e^z \int_0^1 s^n (1-s)^{n+1} \e^{-sz} \, \d s
\end{align}
for all $z \in \IC$ such that $Q_n(z) \neq 0$. 

By a famous result of Ehle \cite[Cor.~3.2]{Ehle}, subdiagonal Pad\'{e} approximations are \emph{$\mathcal{A}$-stable}, i.e.\@ for $n \in \IN$ the function $r_n$ is holomorphic in a neighborhood of the closed left halfplane $\overline{\IC_{-}}$ and $\abs{r_n(z)} \leq 1$ holds for $z\in\overline{\IC_{-}}$. The polynomial $Q_n$ has pairwise distinct roots \cite[Thm.~4.11]{hairer2004solving} and, combining Corollaries 1.1 and 3.7 in \cite{Ehle}, it follows that these roots are located in the open right half-plane.

\section{Technical estimates}
\label{Technical estimates}

\noindent In this section we present a careful and rather technical analysis of the behavior of the subdiagonal Pad\'{e} approximations on the imaginary axis. With a view towards the functional calculi from Section~\ref{functional calculus}, we shall mostly work with $r_n(- \, \cdot)$ rather than with $r_{n}$.

The following lemma, proved recently in \cite{Frank-Lee}, is a key ingredient.

\begin{lemma}
\label{Pade-boundary-estimate}
If $n \in \IN$ and $t \in \IR$ then
\begin{align*}
 \abs{Q_n(\i t)} \geq 1, \quad \bigg |\frac{Q'_n(\i t)}{Q_n(\i t)} \bigg | \leq 1, \quad \text{and} \quad \abs{r'_n(\i t)} \leq 2.
\end{align*}
\end{lemma}

First estimates for the error-term $\abs{r_n(-z) - \e^{-z}}$ and its complex derivative can be obtained from the Perron representation \eqref{Perron}.

\begin{lemma}
\label{error estimate}
If $n \in \IN$ and $z \in \overline{\IC_{+}}$ then
\begin{align*}
 \abs{r_n(-z) - \e^{-z}} \leq \frac{1}{2}\left(\frac{n!}{(2n+1)!}\right)^2 \abs{z}^{2n+2}
\end{align*}
and
\begin{align*}
 \abs{r'_n(-z) - \e^{-z}} \leq \left(\frac{n!}{(2n+1)!}\right)^2 \left(\frac{4}{5} \abs{z}^{2n+2}+ (n+1)\abs{z}^{2n+1} \right).
\end{align*}
\end{lemma}

\begin{proof}
Fix $n \in \IN$ and $z \in \overline{\IC_{+}}$. We first prove the estimate for $\abs{r_n(-z) - \e^{-z}}$. Taking the absolute value in \eqref{Perron} and computing the integral over $s$ using the representation of Euler's beta function yields
 \begin{align}
 \label{Error-Estimate: Eq1}
 \begin{split}
  \abs{r_n(-z) - \e^{-z}} 
  &\leq \frac{1}{\abs{Q_n(-z)}} \frac{1}{(2n+1)!} \abs{z}^{2n+2} \int_0^1 s^n (1-s)^{n+1} \abs{\e^{(s-1)z}} \, \d s \\
  &\leq \frac{1}{\abs{Q_n(-z)}} \frac{1}{(2n+1)!} \abs{z}^{2n+2} \frac{n!(n+1)!}{(2n+2)!}.
 \end{split}
 \end{align}
As $Q_n$ is a polynomial having all roots in the open right halfplane, Lemma~\ref{Pade-boundary-estimate} and the maximum principle for holomorphic functions yield $\frac{1}{\abs{Q_n(-z)}} \leq 1$. Hence \eqref{Error-Estimate: Eq1} implies
\begin{align*}
  \abs{r_n(-z) - \e^{-z}} \leq  \frac{n!(n+1)!}{(2n+1)!(2n+2)!} \abs{z}^{2n+2} = \frac{1}{2}\left(\frac{n!}{(2n+1)!}\right)^2 \abs{z}^{2n+2}.
\end{align*}

For the proof of the second claim we differentiate the Perron representation \eqref{Perron} with respect to $z$, obtaining
\begin{align}
\label{error estimate derivative: split equation}
 r'_n(-z) - \e^{-z} = T_{n,1}(z) + T_{n,2}(z) + T_{n,3}(z),
\end{align}
where
\begin{align*}
 T_{n,1}(z) 	&= \frac{(-1)^{n+3} Q'_n(-z)}{Q_n(-z)^2} \frac{(-z)^{2n+2}}{(2n+1)!} \int_0^1 s^n (1-s)^{n+1} \e^{(s-1)z} \, \d s \\
		&= -\frac{Q'_n(-z)}{Q_n(-z)} (r_n(-z) - \e^{-z}), \\ 
 T_{n,2}(z)	&= \frac{(-1)^{n+2}}{Q_n(-z)} \frac{(2n+2)(-z)^{2n+1}}{(2n+1)!} \int_0^1 s^n (1-s)^{n+1} \e^{(s-1)z}\, \d s \\
		&= -\frac{2n+2}{z} (r_n(-z) - \e^{-z}), \quad \text{and}\\
 T_{n,3}(z) 	&= \frac{(-1)^{n+2}}{Q_n(-z)} \frac{(-z)^{2n+2}}{(2n+1)!} \int_0^1 s^n (1-s)^{n+2} \e^{(s-1)z} \, \d s.		
\end{align*}
We estimate these three terms separately. As all roots of $Q_n$ lie in $\IC_+$, Lemma~\ref{Pade-boundary-estimate} and the maximum principle for holomorphic functions yield $\|\frac{Q'_n(-\, \cdot)}{Q_n(-\, \cdot)}\|_{\HT^{\infty}\!(\IC_{+})} \leq 1$. The first part of this lemma then implies
\begin{align*}
 \abs{T_{n,1}(z)} \leq \frac{1}{2} \bigg(\frac{n!}{(2n+1)!}\bigg)^2 \abs{z}^{2n+2}
\end{align*}
and
\begin{align*}
 \abs{T_{n,2}(z)} \leq (n+1)\bigg(\frac{n!}{(2n+1)!}\bigg)^2 \abs{z}^{2n+1}.
\end{align*}
Finally, note that the only difference between $T_{n,3}(z)$ and $r_n(-z) - \e^{-z}$ is an additional factor $(1-s)$ under the respective integral sign. By the same arguments as in the proof of the first assertion of this lemma, taking into account that $\frac{n+2}{2n+3} \leq \frac{3}{5}$, 
\begin{align*}
 \abs{T_{n,3}(z)} \leq \frac{n!(n+2)!}{(2n+1)!(2n+3)!} \abs{z}^{2n+2} \leq \frac{3}{10} \bigg(\frac{n!}{(2n+1)!}\bigg)^2 \abs{z}^{2n+2}.
\end{align*}
Inserting these inequalities on the right-hand side of \eqref{error estimate derivative: split equation} concludes the proof.
\end{proof}

Clearly, the error-term $r_n(-z) - \e^{-z}$ does not converge to $0$ uniformly in $z \in \IC_+$ as $n \to \infty$. The following lemma shows that the modified error-terms
\begin{align}
\label{important function}
 f_{n,\alpha}: \overline{\IC_+}\setminus \{0\} \to \IC, \quad f_{n,\alpha}(z) := \frac{r_n(-z) - \e^{-z}}{z^\alpha},
\end{align}
where $n \in \IN$ and $\alpha\in(0,2n+2)$, do converge uniformly.

\begin{lemma}
\label{multiplicator estimate}
 If $n \in \IN$ and $\alpha \in (0, 2n+2)$ then
\begin{align*}
 \sup_{z \in \overline{\IC_+}} \abs{f_{n,\alpha}(z)} \leq 2 \left(\frac{n!}{(2n+1)!}\right)^{\frac{\alpha}{n+1}}.
\end{align*}
\end{lemma}

\begin{proof}
The $\mathcal{A}$-stability of $r_n$ implies
\begin{align*}
 \abs{f_{n,\alpha}(z)} = \left| \frac{r_n(-z) - \e^{-z}}{z^\alpha} \right| \leq \frac{2}{\abs{z}^{\alpha}}\quad  \quad (z \in \overline{\IC_+}\setminus\left\{0\right\}),
\end{align*}
hence in combination with Lemma~\ref{error estimate}:
\begin{align}
\label{multiplicator estimate: eq1}
 \sup_{z \in \overline{\IC_+}} \abs{f_{n,\alpha}(z)} \leq \sup_{z \in \overline{\IC_+}} \min \left\{\frac{2}{\abs{z}^{\alpha}}, \frac{1}{2}\left(\frac{n!}{(2n+1)!}\right)^2 \abs{z}^{2n+2-\alpha} \right \}.
\end{align}
Note that $\abs{z}^{-\alpha}$ is a strictly decreasing function and $\abs{z}^{2n+2-\alpha}$ a strictly increasing function of $\abs{z} \in (0,\infty)$. Hence, the supremum on the right-hand side of \eqref{multiplicator estimate: eq1} is actually a maximum that is attained at the value of $\abs{z}$ for which
\begin{align*}
 \frac{2}{\abs{z}^{\alpha}} = \frac{1}{2}\left(\frac{n!}{(2n+1)!}\right)^2 \abs{z}^{2n+2-\alpha},
\end{align*}
i.e.\@ for
\begin{align*}
 \abs{z} = \left( \frac{2(2n+1)!}{n!}\right)^{\frac{1}{n+1}}.
\end{align*}
Inserting such a $z$ in \eqref{multiplicator estimate: eq1} concludes the proof.
\end{proof}

The following two simple calculus lemmas turn out to be quite useful.

\begin{lemma}
\label{integrand estimate}
 If $u,v,U,V>0$ and $0 \leq w \leq \left(\frac{V}{U}\right)^{\frac{1}{u+v}}$ then
 \begin{align*}
  \int_w^\infty \min \left \{Ur^u, Vr^{-v} \right \} \, \frac{\d r}{r} =V\left(\frac{U}{V}\right)^{\frac{v}{u+v}}\frac{u+v}{uv} - \frac{U}{u} w^u.
 \end{align*}
\end{lemma}
\begin{proof}
Let $r_{0}\geq 0$ be such that $Ur_{0}^{u}=Vr_{0}^{-v}$ and note that $r_0 \geq w$. Therefore, 
\begin{align*}
	\int_w^\infty \min \left \{Ur^u, Vr^{-v} \right \} \, \frac{\d r}{r} =\int_w^{r_{0}} Ur^u \, \frac{\d r}{r} +\int_{r_{0}}^\infty Vr^{-v} \, \frac{\d r}{r}.
\end{align*}
Calculating the integrals on the right-hand side and simplifying yields the claim.
\end{proof}

\begin{lemma}
\label{factorial estimate}
If $n \in \IN$ then
\begin{align*}
 \left(\frac{n!}{(2n+1)!}\right)^{\frac{1}{n+1}} \leq \frac{1}{n+1}.
\end{align*}

\begin{proof}
Simply note that
\begin{align*}
 \left(\frac{n!}{(2n+1)!}\right)^{\frac{1}{n+1}}(n+1) = \left(\frac{n! (n+1)^{n+1}}{(2n+1)!}\right)^{\frac{1}{n+1}} \leq 1 \quad \quad (n \in \IN) &\qedhere.
\end{align*}
\end{proof}
\end{lemma}

For the proof of our main result the following $\Ell^{2}$-estimates for the restriction of the modified error-terms $f_{n,\alpha}$ to the imaginary axis are crucial.

\begin{lemma}
\label{L2 estimates}
If $n \in \IN$ and $\alpha \in (\frac{1}{2}, n+\frac{1}{2}]$ then 
\begin{align*}
 \|f_{n,\alpha}(\i \, \cdot)\|_{\Ell^{2}(\IR)}  \leq \frac{4}{\sqrt{2\alpha-1}}(n+1)^{-\alpha+\frac{1}{2}}
\end{align*}
and
\begin{align*}
 \|(f_{n,\alpha}(\i \, \cdot))'\|_{\Ell^{2}(\IR)} 
\leq \left(\frac{8 \alpha}{(2\alpha+1)^{3/2}}+\frac{13^{\alpha}}{10^{\alpha}}\sqrt{\frac{5^{2 \alpha}}{6 \cdot 13^{2\alpha}} + \frac{360}{13(2\alpha - 1)}} \right) (n+1)^{-\alpha + \frac{1}{2}}.
\end{align*}

\end{lemma}

\begin{proof}
First, we prove the estimate for $\|f_{n,\alpha}(\i \, \cdot)\|_{\Ell^{2}(\IR)}$. By Lemma~\ref{error estimate} and $\mathcal{A}$-stability of $r_n$,
\begin{align*}
 \abs{f_{n,\alpha}(\i t)} \leq \min\left\{\frac{1}{2}\left(\frac{n!}{(2n+1)!}\right)^2 \abs{t}^{2n+2 - \alpha}, 2\abs{t}^{-\alpha} \right\} 
\end{align*}
for $t\in\IR\setminus \left\{0\right\}$, hence
\begin{align*}
 \|f_{n,\alpha}(\i \, \cdot)\|_{\Ell^{2}(\IR)}^2 &\leq 2\int_0^\infty \min \left\{\frac{1}{4}\left(\frac{n!}{(2n+1)!}\right)^4 \abs{t}^{4n+5 - 2\alpha}, 4\abs{t}^{-(2\alpha-1)} \right\} \, \frac{\d t}{t}.
\end{align*}
The latter integral fits perfectly into the setting of Lemma~\ref{integrand estimate}, yielding
\begin{align*}
 \|f_{n,\alpha}(\i \, \cdot)\|_{\Ell^{2}(\IR)}^2
 \leq \frac{1}{2\alpha-1} \frac{32(n+1)}{4n+5-2\alpha}\left( \frac{n!}{2(2n+1)!} \right)^{\frac{2\alpha-1}{n+1}}.
\end{align*}
Now, Lemma~\ref{factorial estimate} and $\alpha \leq n+\frac{3}{2}$ imply
\begin{align*}
 \|f_{n,\alpha}(\i \, \cdot)\|_{\Ell^{2}(\IR)}^2
  \leq \frac{16}{2\alpha-1} \left( \frac{n!}{2(2n+1)!} \right)^{\frac{2\alpha-1}{n+1}}
 \leq \frac{16}{2\alpha-1} (n+1)^{-2 \alpha +1}.
\end{align*}
This proves the first estimate. 

For the proof of the second estimate we begin by calculating
\begin{align}
\label{L2 estimate: decomposition}
 \frac{\d}{\d t} f_{n,\alpha}(\i t) = - \i \frac{r'_n(-\i t) - \e^{-\i t}}{(it)^\alpha} - \i \alpha \frac{r_n(-\i t) - \e^{-\i t}}{(it)^{\alpha+1}} =: -\i g_{n,\alpha}(t) - \i \alpha h_{n,\alpha}(t)
\end{align}
for $t \neq 0$. Note that $h_{n,\alpha}(\cdot) = f_{n,\alpha+1}(\i \, \cdot)$. Since we have only used $\alpha \leq n+\frac{3}{2}$ to prove the first part of this lemma, we can apply the first part with $\alpha + 1$ in place of $\alpha$ to find
\begin{align}
\label{estimate hna}
 \|h_{n,\alpha}\|_{\Ell^{2}(\IR)}  \leq \frac{4}{\sqrt{2\alpha+1}}(n+1)^{-\alpha-\frac{1}{2}} \leq \frac{8}{(2\alpha+1)^{3/2}}(n+1)^{-\alpha+\frac{1}{2}},
\end{align}
where the second step is valid since $\alpha \in(\frac{1}{2}, n+\frac{1}{2}]$.

As for $g_{n,\alpha}$, combining Lemma~\ref{Pade-boundary-estimate} and Lemma~\ref{error estimate} yields
\begin{align*}
 \abs{g_{n,\alpha}(t)} \leq \min \bigg \{\bigg(\frac{n!}{(2n+1)!}\bigg)^2 \bigg(\frac{4}{5}\abs{t}^{2n+2-\alpha} + (n+1)\abs{t}^{2n+1-\alpha}\bigg),  3\abs{t}^{-\alpha}  \bigg \} \, \d t
\end{align*}
for $t \neq 0$. Substitute $K:=\big(\frac{n!}{(2n+1)!}\big)^2$ and $L:= n + \frac{9}{5}$. As $\abs{t}^{2n+2-\alpha} \leq \abs{t}^{2n+1-\alpha}$ if and only if $\abs{t} \leq 1$,  the quantity $\|g_{n,\alpha}\|_{\Ell^{2}(\IR)}^2$ can be estimated from above by
\begin{align}\label{integral quantity}
2 \int_0^1 K^2 L^2 t^{4n+2-2\alpha} \, \d t
+ 2 \int_1^\infty \min \Big \{ K^2 L^2 t^{4n+5-2\alpha},  9t^{-(2\alpha-1)}  \Big \} \, \frac{\d t}{t}.
\end{align}
Using Lemma~\ref{factorial estimate},  $\alpha \leq n+\frac{1}{2}$ and $\frac{n + \frac{9}{5}}{n+1} \leq \frac{7}{5}$ in sequence yields
\begin{align}
\label{K2L2 estimate}
 \frac{K^2L^2}{3(n+1)} 
\leq \frac{\left(n+\frac{9}{5}\right)^2}{3(n+1)^{4n+5}}
\leq \frac{2}{3(n+1)^{4\alpha + 1}}
\leq \frac{1}{3 \cdot 2^{2\alpha+1}(n+1)^{2\alpha - 1}}.
\end{align}
In particular, $9K^{-2}L^{-2} \geq \tfrac{9}{2}(n+1)^{4\alpha} \geq 1$, which in turn allows us to compute the second integral in \eqref{integral quantity} by means of Lemma~\ref{integrand estimate}. This yields
\begin{align*}
\|g_{n,\alpha}\|_{\Ell^{2}(\IR)}^2
\leq \frac{2K^2L^2}{4n+3-2\alpha} + \frac{18}{2\alpha-1} \bigg(\frac{KL}{3}\bigg)^{\frac{2\alpha - 1}{2n+2}} \frac{4n+4}{4n+5-2\alpha}-\frac{2K^{2}\Ell^{2}}{4n+5-2\alpha},
\end{align*}
which after simplifying, using $\alpha \in(\frac{1}{2},n+\frac{1}{2}]$ and \eqref{K2L2 estimate}, allows to estimate $\|g_{n,\alpha}\|_{\Ell^{2}(\IR)}^2$ from above by
\begin{align*}
\frac{K^2L^2}{3(n+1)} + \frac{36}{2\alpha -1} \bigg(\frac{KL}{3}\bigg)^{\frac{2\alpha - 1}{2n+2}}\leq \frac{1}{3\cdot 2^{2\alpha+1}}n^{-2\alpha+1}+\frac{36}{2\alpha -1} \bigg(\frac{KL}{3}\bigg)^{\frac{2\alpha - 1}{2n+2}}.
\end{align*}
For the second term, Lemma~\ref{factorial estimate} and the inequality $n + \frac{9}{5} \leq \left(\frac{13}{10}\right)^{2n+2}$, which can easily be verified by induction, yield
\begin{align*}
 \bigg(\frac{KL}{3}\bigg)^{\frac{1}{2n+2}} = \bigg(\frac{n!}{\sqrt{3}(2n+1)!}\bigg)^{\frac{1}{n+1}} \bigg(n + \frac{9}{5}\bigg)^{\frac{1}{2n+2}} \leq \frac{13}{10} \frac{1}{n+1}.
\end{align*}
Hence
\begin{align*}
 \|g_{n,\alpha}\|_{\Ell^{2}(\IR)}^2 
 &\leq \frac{1}{3 \cdot 2^{2\alpha+1}}(n+1)^{-2\alpha+1}+\frac{36}{2\alpha-1}\left(\frac{13}{10}\right)^{2\alpha-1}(n+1)^{-2\alpha+1}\\
 &= \frac{13^{2\alpha}}{10^{2 \alpha}}\left(\frac{5^{2 \alpha}}{6 \cdot 13^{2\alpha}} + \frac{360}{13(2\alpha - 1)}\right) (n+1)^{-2\alpha + 1}.
\end{align*}
The conclusion follows by combining the estimate above, \eqref{estimate hna} and \eqref{L2 estimate: decomposition}.
\end{proof}

\section{Precise formulation and proof of the main results}
\label{Proofs}

\subsection{Uniformly bounded semigroups}

The following is our main result for uniformly bounded semigroups.

\begin{theorem}[Convergence for uniformly bounded semigroups]
\label{main result uniformly bounded}
Let $(T(t))_{t \geq 0}$ be a $C_{0}$-semigroup of type $(M,0)$ on a Banach space $X$, with generator $-A$. Let $\alpha > \frac{1}{2}$ and $x\in \D(A^\alpha)$ be given. Then 
\begin{align}
\label{estimate uniformly bounded}
\left \|r_n(-tA)x - T(t)x \right \| \leq C(\alpha) M t^{\alpha} (n+1)^{-\alpha + \frac{1}{2}} \|A^\alpha x\|
\end{align}
for all $t\geq 0$ and $n\in\IN$ with $n\geq \alpha-\frac{1}{2}$. Here, 
\begin{align}
\label{C(alpha)}
 C(\alpha) := \frac{\sqrt{2}}{(2\alpha-1)^{1/4}}\sqrt{\frac{8 \alpha}{(2\alpha+1)^{3/2}}+\frac{13^{\alpha}}{10^{\alpha}}\sqrt{\frac{5^{2 \alpha}}{6 \cdot 13^{2\alpha}} + \frac{360}{13(2\alpha - 1)}}}.
\end{align}

In particular, for each $\alpha>\frac{1}{2}$ the sequence $(r_n(-tA))_{n \in \IN}$ converges strongly on $\D(A^\alpha)$ and locally uniformly in $t\geq 0$ to $T(t)$ with rate $\mathcal{O}(n^{-\alpha + \frac{1}{2}})$.
\end{theorem}

Having at hand the technical estimates from Section~\ref{Technical estimates}, the main step towards a proof of Theorem~\ref{main result uniformly bounded} is to identify the modified error term as the Laplace transform of an $\Ell^{1}(\IR_+)$-function provided that $\alpha > \frac{1}{2}$. The proof of Theorem~\ref{main result uniformly bounded} will then be a straightforward application of the Hille-Phillips calculus.

\begin{lemma}
\label{paley-wiener lemma}
Let $f\in\HT^{\infty}\!(\IC_{+}) \cap \Ce(\overline{\IC_+})$ be such that $f(z) \in \mathcal{O}(\abs{z}^{-a})$ as $\abs{z} \to \infty$ for some $a>\frac{1}{2}$. Then there is a function $g \in \Ell^{2}(\IR_{+})$ such that 
\begin{align*}
 f = \hat{g} \quad \text{and} \quad \|g\|_{\Ell^{2}(\IR_{+})} = \frac{1}{\sqrt{2 \pi}}\|f(\i \, \cdot)\|_{\Ell^{2}(\IR)}.
\end{align*}
If in addition $f(\i \, \cdot)$ is differentiable and the derivative $(f(\i \, \cdot))'$ belongs to $\Ell^{2}(\IR)$, then $g \in \Ell^{1}(\IR_+)$ with 
\begin{align*}
 \|g\|_{\Ell^{1}(\IR_+)} \leq \frac{1}{\sqrt{2}} \|f(\i \, \cdot)\|_{\Ell^{2}(\IR)}^{\frac{1}{2}} \|(f(\i \, \cdot))'\|_{\Ell^{2}(\IR)}^{\frac{1}{2}}.
\end{align*}
\end{lemma}

\begin{proof}
From the assumption that $f(z)\in \mathcal{O}(\abs{z}^{-a})$ as $\abs{z} \to \infty$ it follows that
\begin{align*}
 f\in\HT^2(\IC_+) := \Big \{h: \IC_+ \to \IC \, \text{hol.} \, \Big | \, \|h\|_{\HT^2(\IC_+)} := \sup_{x>0} \frac{1}{\sqrt{2 \pi}} \|h(x + \i \, \cdot)\|_{\Ell^{2}(\IR)} < \infty \Big \}.
\end{align*}
Hence the first assertion follows from the Paley-Wiener Theorem \cite[Thm.~19.2]{RudinRC}. Also note that the extension of $g$ to the whole real axis by zero is $\mathcal{F}^{-1}(f(\i \, \cdot))$.

Now, suppose in addition that $h(\cdot):=f(\i \, \cdot)$ is differentiable with derivative in $\Ell^{2}(\IR)$. Plancherel's theorem implies
\begin{align*}
 \mathcal{F}^{-1}h \in \Ell^{2}(\IR) \quad \text{and} \quad \Big(\xi \mapsto -\i \xi (\mathcal{F}^{-1}h)(\xi) = (\mathcal{F}^{-1}h')(\xi)\Big) \in \Ell^{2}(\IR).
\end{align*}
Therefore Carlson's inequality \cite[p.175]{Beckenbach-Bellman} yields $\mathcal{F}^{-1}h \in \Ell^{1}(\IR)$ with the estimate
\begin{align*}
 \|\mathcal{F}^{-1}h \|_{\Ell^{1}(\IR)} \leq \sqrt{\pi} \|\mathcal{F}^{-1}h\|_{\Ell^{2}(\IR)}^{\frac{1}{2}} \|\mathcal{F}^{-1}h'\|_{\Ell^{2}(\IR)}^{\frac{1}{2}}.
\end{align*}
Thus by definition of $h$ and Plancherel's theorem
\begin{align*}
 \|g\|_{\Ell^{1}(\IR_{+})}=\|\mathcal{F}^{-1}(f(\i \, \cdot)) \|_{\Ell^{1}(\IR)} \leq \frac{1}{\sqrt{2}}\|f(\i \, \cdot)\|_{\Ell^{2}(\IR)}^{\frac{1}{2}} \|(f(\i \, \cdot))'\|_{\Ell^{2}(\IR)}^{\frac{1}{2}}. & \qedhere
\end{align*}
\end{proof}

\begin{remark}
Lemma~\ref{paley-wiener lemma} is very similar to Lemmas 1 and 2 in \cite{Brenner-Thomee}. However, note that our version is tailored to the Laplace transform since it uses $\supp(\mathcal{F}^{-1}h) \subseteq \IR_+$. This results in a constant smaller by a factor of $\frac{1}{2}$ in Carlson's inequality and consequently also in the $\Ell^{1}$-estimate for $g$.
\end{remark}

\begin{proof}[Proof of Theorem~\ref{main result uniformly bounded}]
Fix $t$ and $n$ as required. Note that $z\mapsto r_{n}(-tz)$ is the Laplace transform of a bounded measure by Lemma~\ref{paley-wiener lemma} and that $z\mapsto e^{-tz}$ is the Laplace transform of unit point mass at $t$. Hence $r_{n}(-tA)=r_{n}(-t \, \cdot)(A)$ and $\e^{-t \, \cdot}(A)=T(t)$ are well-defined in the Phillips calculus for $A$. Set
\begin{align}\label{function definition}
 f(z):= t^\alpha f_{n,\alpha}(tz)=\frac{r_n(-tz) - \e^{-tz}}{z^\alpha} \quad \quad (z \in \overline{\IC_+}\setminus\{0\})
\end{align}
with $f_{n,\alpha}$ as in \eqref{important function}. Then $f \in \HT^{\infty}\!(\IC_{+}) \cap \Ce(\overline{\IC_+})$ by Lemma~\ref{multiplicator estimate} and $f(z) \in \mathcal{O}(\abs{z}^{-\alpha})$ as $\abs{z}~\to~\infty$ thanks to the $\mathcal{A}$-stability of $r_n$. Moreover, $f(\i \, \cdot)$ is differentiable, with derivative in $\Ell^{2}(\IR)$ by Lemma~\ref{L2 estimates}. Thus Lemma~\ref{paley-wiener lemma} yields a function $g \in \Ell^{1}(\IR_{+})$ such that $f = \hat{g}$ and
\begin{align*}
 \|g\|_{\Ell^{1}(\IR_+)} \leq \frac{1}{\sqrt{2}} \|f(\i \, \cdot)\|_{\Ell^{2}(\IR)}^{\frac{1}{2}} \|(f(\i \, \cdot))'\|_{\Ell^{2}(\IR)}^{\frac{1}{2}} = \frac{t^\alpha}{\sqrt{2}} \|f_{n,\alpha}(\i \, \cdot)\|_{\Ell^{2}(\IR)}^{\frac{1}{2}} \|(f_{n,\alpha}(\i \, \cdot))'\|_{\Ell^{2}(\IR)}^{\frac{1}{2}}.
\end{align*}
By the Hille-Phillips calculus,
\begin{align}
\label{HP representation error}
(r_n(-tA) - T(t))x = f(A)A^\alpha x = \hat{g}(A)A^\alpha x = \int_0^\infty g(s)T(s)A^\alpha x \, \d s.
\end{align}
Hence, by taking norms,
\begin{align*}
\|r_n(-tA)x - T(t)x\| \leq \frac{M t^\alpha}{\sqrt{2}} \|f_{n,\alpha}(\i \, \cdot)\|_{\Ell^{2}(\IR)}^{\frac{1}{2}} \|(f_{n,\alpha}(\i \, \cdot))'\|_{\Ell^{2}(\IR)}^{\frac{1}{2}} \|A^\alpha x\|.
\end{align*}
The conclusion follows by substituting the estimates from Lemma~\ref{Pade-boundary-estimate} for the $\Ell^{2}(\IR)$-norms on the right-hand side of the inequality above.
\end{proof}

\begin{remark}
\label{Special values of C(alpha)}
With a view towards the problem from the introduction it might be interesting to have a rough estimate of the size of $C(\alpha)$ for $\alpha = 1$ and other small values. As can be shown by explicit computations,
\begin{align*}
 C(1) \leq 4.10, \quad C(2) \leq 2.76, \quad C(3) \leq 2.41, \quad C(4)\leq 2.28.
\end{align*}
\end{remark}

\begin{remark}
Suppose that, in the setting of Theorem~\ref{main result uniformly bounded}, $T=(T(t))_{t \geq 0}$ is exponentially stable. Proceeding similar to the proof above we can, depending on the type of $T$, provide a sharper upper bound for the approximation error. To be precise, let $T$ be of type $(M,-\omega)$ for some $\omega > 0$. With $f$, $f_{n,\alpha}$ and $g$ as in the proof of Theorem~\ref{main result uniformly bounded}, Lemma~\ref{paley-wiener lemma} implies
\begin{align*}
 \|g\|_{\Ell^{2}(\IR_+)} = \frac{1}{\sqrt{2 \pi}}\|f(\i \, \cdot)\|_{\Ell^{2}(\IR)} = \frac{t^{\alpha - \frac{1}{2}}}{\sqrt{2 \pi}} \|f_{n,\alpha}(\i \, \cdot)\|_{\Ell^{2}(\IR)}.
\end{align*}
Taking norms in \eqref{HP representation error} and applying H\"older's inequality yields
\begin{align*}
 \|r_n(-tA)x - T(t)x\| \leq M\|ge^{-\omega\cdot}\|_{\Ell^{1}(\IR)}\leq \frac{M}{\sqrt{4 \pi \omega}} \frac{4}{\sqrt{2\alpha-1}} t^{\alpha - \frac{1}{2}} (n+1)^{-\alpha+\frac{1}{2}} \|A^\alpha x\|
\end{align*}
for all $\alpha > \frac{1}{2}$, $n \geq \alpha - \frac{1}{2}$, $t \geq 0$ and $x \in \D(A^\alpha)$.
\end{remark}

\subsection{Analytic semigroups}

For $A$ a sectorial operator of angle $\varphi\in(0,\pi)$ and $\nu\in(\varphi,\pi)$, let
\begin{align}
\label{A sectorial}
 M_{\nu}:=\sup_{\lambda \in \partial \Sigma_\nu} \|\lambda R(\lambda,A)\|<\infty.
\end{align}
For bounded analytic semigroups the following improvement of Theorem~\ref{main result uniformly bounded} holds.

\begin{theorem}[Convergence for bounded analytic semigroups]
\label{1st theorem}
Let $A$ be a sectorial operator of angle $\varphi\in(0,\frac{\pi}{2})$ on a Banach space $X$ and let $(T(t))_{t \geq 0}$ be the bounded analytic $C_{0}$-semigroup generated by $-A$. Let $\alpha >0$ and $x\in\D(A^{\alpha})$ be given. Then 
\begin{align*}
\left \|r_n(-tA)x - T(t)x \right \| \leq \frac{4M_{\nu}}{\alpha\pi} t^{\alpha} (n+1)^{-\alpha} \|A^\alpha x\|
\end{align*}
for all $\nu \in (\varphi, \frac{\pi}{2})$, $t \geq 0$ and $n \in \IN$ such that $n \geq \alpha - 1$.

In particular, for each $\alpha>0$ the sequence $(r_n(-tA))_{n \in \IN}$ converges strongly on $\D(A^\alpha)$ and locally uniformly in $t\geq 0$ to $T(t)$ with rate $\mathcal{O}(n^{-\alpha})$.
\end{theorem}

\begin{proof}
Fix $\nu\in(\varphi,\frac{\pi}{2})$ and $n \geq \alpha - 1$. Observe that for $t=0$ the statement is trivial, whereas for $t>0$ the operator $tA$ is also sectorial of angle $\varphi$. Hence, by replacing $A$ by $tA$ and noting that the value of $M_{\nu}$ does not change under this replacement, it suffices to give a proof for $t = 1$.

To this end, let $\psi \in (\nu, \frac{\pi}{2})$ and let $f_{n,\alpha}$ be as in \eqref{important function}. Lemma~\ref{error estimate} and $\mathcal{A}$-stability of $r_n$ yield
\begin{align}
\label{Th1:Eq1}
 \abs{f_{n,\alpha}(z)} \leq \min\left\{\frac{1}{2}\left(\frac{n!}{(2n+1)!}\right)^2 \abs{z}^{2n+2 - \alpha}, 2\abs{z}^{-\alpha} \right\}\quad \quad (z \in \IC_{+}),
\end{align}
so that $f_{n,\alpha}\in \HT_0^\infty(\Sigma_\psi)$. By the holomorphic functional calculus for sectorial operators,
\begin{align*}
 \left(r_n(-A) - T(1)\right)x = f_{n,\alpha}(A)A^{\alpha}x = \frac{1}{2 \pi \i} \int_{\partial \Sigma_\nu} f_{n,\alpha}(z)R(z,A)A^{\alpha}x \, \d z,
\end{align*}
where $\partial \Sigma_\psi$ is oriented such that $\sigma(A)$ is surrounded counterclockwise. Taking operator norms and estimating the integrand on the right-hand side by means of \eqref{A sectorial} and \eqref{Th1:Eq1} yields
\begin{align*}
  \left\|\left(r_n(-A) - T(1)\right)x\right\| \leq \frac{M_{\nu}}{\pi}\|A^{\alpha}x\| \int_0^\infty \min \left\{\frac{1}{2}\left(\frac{n!}{(2n+1)!}\right)^2 r^{2n+2 - \alpha}, 2r^{-\alpha} \right\} \, \frac{\d r}{r},
\end{align*}
which fits perfectly into the setting of Lemma~\ref{integrand estimate}. Hence,
\begin{align*}
 \left\|\left(r_n(-A) - T(1)\right)x\right\| 
 &\leq \frac{4M_{\nu}}{\alpha\pi} \frac{n+1}{2n+2-\alpha}\left( \frac{n!}{2(2n+1)!} \right)^{\frac{\alpha}{n+1}}\|A^{\alpha}x\| \\
 &\leq \frac{4M_{\nu}}{\alpha\pi} \left( \frac{n!}{(2n+1)!} \right)^{\frac{\alpha}{n+1}}\|A^{\alpha}x\|.
\end{align*}
The rightmost term can be estimated by means of Lemma~\ref{factorial estimate}. This leads to
\begin{align*}
 \left\|\left(r_n(-A) - T(1)\right)x\right\| \leq \frac{4M_{\nu}}{\alpha\pi} (n+1)^{-\alpha}\|A^{\alpha}x\|,
\end{align*}
which is the required estimate for $t=1$.
\end{proof}

\begin{remark}
In the situation of Theorem~\ref{1st theorem} the scaling and squaring methods associated to a fixed subdiagonal Pad\'{e} approximant converge strongly on $X$ and even in $\mathcal{L}(X)$, see \cite[Thm.~4.4]{Larsson-Thomee}. Whether this is true for the method without scaling and squaring as well is left as an open problem. In Section~\ref{Hinfty calculus} we will prove strong convergence on $X$ under an additional assumption on $A$. 
\end{remark}

\subsection{Exponentially $\gamma$-stable semigroups}

In this section we strengthen the convergence result from Theorem~\ref{main result uniformly bounded} for so-called $\gamma$-bounded $C_{0}$-semigroups. The notion of $\gamma$-boundedness, originating in the work of Kalton and Weis \cite{Kalton-Weis}, allows to generalize results which rely on Plancherel's theorem, and therefore on a Hilbert space structure, to general Banach spaces. The usefulness of the concept arises from the fact that many families of operators can be shown to be $\gamma$-bounded. It is closely related to the notion of $R$-boundedness that has proved to be essential in e.g.\@ questions of maximal regularity \cite{Weis}, and in fact the two notions coincide on many common spaces. For more background on $\gamma$-boundedness see \cite{vanNeerven}.

Given a Banach space $X$, a collection $\mathcal{T} \subseteq \mathcal{L}(X)$ is \emph{$\gamma$-bounded} if there is a constant $C \geq 0$ such that
\begin{align*}
 \mathbb{E} \left( \bigg \| \sum_{T \in \mathcal{T}'} \gamma_T Tx_T \bigg \|_X^2 \right)^{\frac{1}{2}} \leq 
 C \mathbb{E} \left( \bigg \| \sum_{T \in \mathcal{T}'} \gamma_T x_T \bigg \|_X^2 \right)^{\frac{1}{2}},
\end{align*}
for all finite subsets $\mathcal{T}' \subseteq \mathcal{T}$, all $x_T \in X$ and all sequences of independent standard Gaussian random variables $\gamma_T$ on some probability space. As usual, $\mathbb{E}$ denotes the expectation value. The smallest such constant $C$ is called the \emph{$\gamma$-bound} of $\mathcal{T}$ and is denoted by $\llbracket\mathcal{T}\rrbracket^\gamma$. 

Gamma-bounded collections are uniformly bounded (take $\mathcal{T}' = \{T\}$ for $T \in \mathcal{T}$) but the converse is not true in general. However, it is on Hilbert spaces, as follows by writing the norm on $X$ via the inner product and applying the independence of the respective random variables.

\begin{lemma}
\label{hilbert space case}
Let $X$ be a Hilbert space. A collection $\mathcal{T} \subseteq \mathcal{L}(X)$ is $\gamma$-bounded if and only if it is uniformly bounded, with $\gamma$-bound equal to the uniform bound.
\end{lemma}

Now, we specialize to $C_{0}$-semigroups. Note that a $C_{0}$-semigroup $(T(t))_{t\geq 0}$ has type $(M,\omega)$ if and only if $\sup\left\{\|e^{-\omega t}T(t)\|\mid t\geq 0\right\}\leq M$. Together with Lemma~\ref{hilbert space case}, this shows that type and $\gamma$-type defined below coincide for $C_{0}$-semigroups on Hilbert spaces.

\begin{definition}[$\gamma$-bounded semigroups]
\label{gamma-bounded}
A $C_{0}$-semigroup $T = (T(t))_{t \geq 0}$ is said to have \emph{$\gamma$-type} $(M,\omega) \in [1,\infty) \times \IR$ if $\llbracket\e^{-\omega t}T(t) \mid t \geq 0\rrbracket^\gamma \leq M$. The \emph{exponential $\gamma$-bound} of $T$ is defined as
\begin{align*}
 \omega_\gamma(T) := \inf \{\omega \mid \text{$T$ is of $\gamma$-type $(M, \omega)$} \}\in[-\infty,\infty],
\end{align*}
and $T$ is \emph{exponentially $\gamma$-stable} if $\omega_\gamma(T) < 0$.
\end{definition}

The connection to the kind of results we are after is the following proposition, which follows from \cite[Corollary 6.3]{Haase-Rozendaal} by scaling the semigroup.

\begin{proposition}
\label{rozendaal result}
Let $-A$ generate a $C_{0}$-semigroup of $\gamma$-type $(M,\omega) \in [1,\infty) \times (-\infty,0)$ on a Banach space $X$ and let $\beta>0$. Then there is a constant $C = C(M,\omega,\beta)$ such that $f(A)A^{-\beta} \in \mathcal{L}(X)$ and
\begin{align*}
 \|f(A)A^{-\beta}\| \leq C \|f\|_{\HT^{\infty}\!(\IC_{+})}
\end{align*}
for all $f\in\HT^{\infty}\!(\IC_{+})$.
\end{proposition}

For exponentially $\gamma$-stable $C_{0}$-semigroups we now derive an extension of Theorem~\ref{main result uniformly bounded}. Note that we also obtain a better rate of convergence on the subspaces covered by Theorem~\ref{main result uniformly bounded}.

\begin{theorem}[Convergence for exponentially $\gamma$-stable semigroups]
\label{main result gamma-stable}
Let $-A$ generate a $C_{0}$-semigroup $T=(T(t))_{t \geq 0}$ on a Banach space $X$ and suppose that $T$ has $\gamma$-type $(M,-\omega)$ for certain $M\geq 1$ and $\omega>0$. Let $\alpha>0$, $a\in (0,\alpha)$ and $x\in \D(A^{\alpha})$ be given. Then there is a constant $C = C(M, \omega, \alpha-a)$ such that
\begin{align*}
 \left \|r_n(-tA)x - T(t)x \right \| \leq C t^{a} (n+1)^{-a} \|A^{\alpha} x\|
\end{align*}
for all $t \geq 0$ and all $n \in \IN$ such that $n > \frac{a}{2} - 1$. 

In particular, for each $\alpha>0$ the sequence $(r_n(-tA))_{n \in \IN}$ converges strongly on $\D(A^{\alpha})$ and locally uniformly in $t\geq 0$ to $T(t)$ with rate $\bigcap_{a < \alpha} \mathcal{O}(n^{-a})$. 
\end{theorem}

\begin{proof}
The proof is very similar to that of Theorem~\ref{main result uniformly bounded}, appealing to Proposition~\ref{rozendaal result} instead of Lemma~\ref{paley-wiener lemma}. Fix $t$ and $n$ as required and let $f$ be as in \eqref{function definition} with $\alpha$ replaced by $a$. Proposition~\ref{rozendaal result} yields
\begin{align*}
 \|(r_n(-tA) - T(t))A^{-\alpha}\| = \|f(A) A^{a-\alpha}\| \leq C(M,\omega,\alpha-a) \|f\|_{\HT^{\infty}\!(\IC_{+})}.
\end{align*}
Lemma~\ref{multiplicator estimate} and Lemma~\ref{factorial estimate} imply
\begin{align*}
\|f\|_{\HT^{\infty}\!(\IC_{+})}
 = \sup_{z \in \IC_+} t^{a}\left|\frac{r_n(-tz) - \e^{-tz}}{(tz)^a}\right|
 \leq 2 t^{a} \left( \frac{n!}{(2n+1)!}\right)^{\frac{a}{n+1}}
 \leq 2 t^{a} (n+1)^{-a}.
\end{align*}
Combining our previous two estimates yields
\begin{align*}
 \|(r_n(-tA) - T(t))x\| 
 = \|(r_n(-tA) - T(t))A^{-\alpha}A^{\alpha}x\|
 \leq  C t^{a} (n+1)^{-a} \|A^{\alpha}x\|
\end{align*}
for some constant $C=C(M, \omega, \alpha- a)$.
\end{proof}

\begin{remark}
\label{Hilbert space result}
By Lemma~\ref{hilbert space case}, the conclusion of Theorem~\ref{main result gamma-stable} specifically holds true if $T$ is an exponentially stable $C_{0}$-semigroup on a Hilbert space.
\end{remark}

\begin{remark}
\label{group case}
Theorem~\ref{main result gamma-stable} yields a convergence rate $\mathcal{O}(n^{-a})$ on $\D(A^\alpha)$ for arbitrary $a<\alpha$ but it does not provide a statement concerning the limit case $a=\alpha$. Suppose, in addition to the hypotheses from Theorem~\ref{main result gamma-stable}, that the operators $T(t)$, $t \geq 0$, are invertible and the collection $\left\{e^{\omega' t}T(t)^{-1}\mid t\geq 0\right\}$ is $\gamma$-bounded for some $\omega'\in\IR$, i.e.\@ that $T$ extends to an exponentially $\gamma$-bounded \emph{group}. In this case the $\gamma$-version of the Boyadzhiev-de Laubenfels Theorem due to Kalton and Weis \cite[Theorem 6.5]{Haase3} implies that $A$ has a bounded $\HT^{\infty}$-calculus on $\IC_{+}$. Hence, Theorem~\ref{main result hinfty} from the next section implies convergence of order $\mathcal{O}(n^{-\alpha})$ on $\D(A^{\alpha})$ and even strong convergence on the whole space $X$. Due to Lemma~\ref{hilbert space case} this applies in particular if $T$ is an exponentially stable semigroup on a Hilbert space for which all operators $T(t)$, $t \geq 0$, are invertible.
\end{remark}

\subsection{Semigroup generators with a bounded $\HT^{\infty}$-calculus}
\label{Hinfty calculus}
Let $-A$ be the generator of a uniformly bounded semigroup $T=(T(t))_{t \geq 0}$ on a Banach space $X$. In all of our results so far, the convergence
\begin{align*}
 r_n(-tA)x \stackrel{n \to \infty}{\longrightarrow} T(t)x
\end{align*}
holds for $x$ belonging to some proper subspace of $X$, if $A$ is unbounded. In this section we show that this convergence can be extended to all $x \in X$ if $A$ has a \emph{bounded $\HT^{\infty}$-calculus} on $\IC_+$, i.e.\@ there is a constant $C\geq 0$ such that
\begin{align}
\label{bounded functional calculus}
\|r(A)\|_{\mathcal{L}(X)}\leq C\|r\|_{\HT^{\infty}\!(\IC_{+})}
\end{align}
holds for all rational functions $r\in\HT^{\infty}\!(\IC_{+})$. The smallest such constant is the \emph{$\HT^{\infty}$-bound} of $A$. For such $A$ the following result holds.

\begin{theorem}
\label{main result hinfty}
Let $(T(t))_{t \geq 0}$ be a uniformly bounded $C_{0}$-semigroup on a Banach space $X$ with generator $-A$ and suppose that $A$ admits a bounded $\HT^{\infty}$-calculus on $\IC_+$ with $\HT^{\infty}$-bound $C$. Let $\alpha >0$ and $x\in \D(A^\alpha)$ be given. Then
\begin{align}
\label{estimate hinfty}
\left \|r_n(-tA)x - T(t)x \right \| \leq  2 C t^{\alpha} (n+1)^{-\alpha} \|A^\alpha x\|
\end{align}
for all $t \geq 0$ and all $n \in \IN$ such that $n > \frac{\alpha}{2} - 1$.

In particular, for each $\alpha > 0$ the sequence $(r_n(-tA))_{n \in \IN}$ converges to $T(t)$ strongly on $\D(A^\alpha)$ with rate $\mathcal{O}(n^{-\alpha})$ and locally uniformly in $t \geq 0$.

Moreover, $(r_{n}(-tA))_{n\in\IN}$ converges strongly to $T(t)$ on $X$, locally uniformly in $t\geq 0$.
\end{theorem}

\begin{proof}
First note that, under the present assumptions, one can extend the Hille-Phillips calculus for $A$ to all functions $f\in \HT^{\infty}\!(\IC_{+})\cap \Ce(\overline{\IC_{+}})$ by taking uniform limits of rational functions \cite[Proposition F.3]{Haase} and then regularizing. This yields a proper functional calculus in the terminology of \cite[Sect.~1.2]{Haase} and \eqref{bounded functional calculus} extends to all $f\in \HT^{\infty}\!(\IC_{+})\cap \Ce(\overline{\IC_{+}})$ for which $\lim_{z\rightarrow \infty}f(z)$ exists. 

Now, fix $t$ and $n$ as required and let $f$ be as in \eqref{function definition}. Then $f \in \HT^{\infty}\!(\IC_{+})\cap \Ce(\overline{\IC_{+}})$ by Lemma \ref{error estimate} and $\lim_{z\rightarrow \infty}f(z)=0$ by $\mathcal{A}$-stability of the $r_{n}$. Hence,
\begin{align*}
 \|r_n(-tA)x - T(t)x\| = \|f(A) A^{\alpha}x\| \leq C \|f\|_{\HT^{\infty}\!(\IC_{+})} \|A^\alpha x\|
\end{align*}
and \eqref{estimate hinfty} follows by estimating $\|f\|_{\HT^{\infty}\!(\IC_{+})}$ via Lemma~\ref{multiplicator estimate} and Lemma~\ref{factorial estimate}. 

Finally, we prove that $(r_{n}(-tA))_{n\in\IN}$ converges strongly to $T(t)$ on $X$, locally uniformly in $t\geq 0$. To this end note that, for $K$ a bounded subset of $\IR_+$, the family 
\begin{align*}
 \{r_n(-tA)-T(t) \mid n \in \IN, t \in K\} \subseteq \mathcal{L}(X)
\end{align*}
is bounded due to \eqref{bounded functional calculus}. Since $r_n(-tA)$ converges to $T(t)$ strongly on $\D(A)$ and uniformly in $t \in K$ as $n \to \infty$, the denseness of $\D(A)$ in $X$ implies that this convergence extends to all of $X$.
\end{proof}

\begin{remark}
\label{contraction semigroups}
The requirements for Theorem~\ref{main result hinfty} are e.g.\@ satisfied if $-A$ generates a $C_{0}$-semigroup on $X$ that is (similar to) a contraction semigroup on a Hilbert space, cf.~Theorem 7.1.7 and Remark 7.1.9 in \cite{Haase}. In that case, one can choose $C=1$ in \eqref{estimate hinfty}. Any bounded $C_{0}$-group on a Hilbert space is similar to a contraction group, cf.~\cite{Sz.-Nagy}.

Theorem~\ref{main result hinfty} also applies to sectorial operators of angle $\varphi<\frac{\pi}{2}$ on a Hilbert space that satisfy certain square function estimates, cf.~\cite{Cowling-Doust-McIntosh-Yagi}.
\end{remark}

\subsection{Extension to intermediate spaces}
\label{intermediate spaces}

In this section we outline how to extend the results from the previous sections to classes of intermediate spaces. Throughout, let $-A$ be the generator of a uniformly bounded $C_{0}$-semigroup $T = (T(t))_{t \geq 0}$ on a Banach space $X$. For $k \in \IN$ the $k$-th \emph{Favard space} is
\begin{align*}
 \mathrm{F}_{k}:=\bigg\{x\in \D(A^{k-1})\mid L(A^{k-1}x):=\limsup_{t\downarrow 0}\frac{1}{t}\|T(t)A^{k-1}x-A^{k-1}x\|<\infty\bigg\}.
\end{align*}
Then $\D(A^k) \subseteq \mathrm{F}_k$ and for non-reflexive Banach spaces this inclusion can be strict, see e.g.\@ Section~\ref{Laplace transform} below. It is therefore remarkable that all convergence results from the previous sections immediately extend from $\D(A^k)$ to $\mathrm{F}_k$ upon replacing $\|A^k x\|$ by $L(A^{k-1}x)$ on the respective right-hand sides. This is due to the subsequent lemma \cite[Prop.\@ 1]{Kovacs}, the short proof of which is included for the reader's convenience.

\begin{lemma}
\label{Kovacs-lemma}
Let $S \in \mathcal{L}(X)$ and suppose there exist $k \in \IN$ and $C\geq 0$ such that $\|Sx\| \leq C\|A^k x\|$ holds for all $x \in \D(A^k)$. Then $\|Sx\| \leq C L(A^{k-1}x)$ holds for all $x \in \mathrm{F}_k$.
\end{lemma}

\begin{proof}
Approximate $x\in \mathrm{F}_{k}$ by elements 
\begin{align*}
x_{t}:=\frac{1}{t}\int_{0}^{t}T(s)x\,\d s\in \D(A^{k}) \quad \quad (t>0)
\end{align*}
and note that
\begin{align*}
 \|Sx_t\| \leq C \|A^{k}x_{t} \| = C\bigg \|\frac{1}{t}(T(t)A^{k-1}x-A^{k-1}x)\bigg \| \quad \quad (t>0).
\end{align*}
Now, the conclusion follows by passing to the limit superior as $t \downarrow 0$.
\end{proof}

Using the inclusions from Proposition~3.1.1, Corollary~6.6.3 and Proposition~B.3.5 in \cite{Haase}, all our convergence results carry over to the domains of certain complex fractional powers, as well as to real and complex interpolation spaces. This includes Favard spaces of non-integer order\cite[Section 3.3]{Kovacs-Diss}.

\section{Application to the inversion of the Laplace transform}
\label{Laplace transform}

\noindent Following an idea from \cite{NOS}, we show how the results from Section \ref{Proofs} can be used to obtain inversion formulas for the vector-valued Laplace transform, with precise error-estimates. 

Throughout, let $X$ be a Banach space and $\Ce_{\textrm{ub}}(\IR_+;X)$ the space of bounded uniformly continuous functions from $\IR_+$ to $X$ equipped with the supremum norm. For $k \in \IN$ denote by $\Ce_{\textrm{ub}}^k(\IR_+;X) \subseteq \Ce_{\textrm{ub}}(\IR_+;X)$ the subspace of $k$-times differentiable functions whose derivatives up to order $k$ belong to $\Ce_{\textrm{ub}}(\IR_+;X)$ and by $\Ce_{\textrm{ub}}^{k,1}(\IR_{+};X)$ the space of all functions $f\in \Ce_{\textrm{ub}}^{k}(\IR_+;X)$ for which $f^{(k)}$ is globally Lipschitz continuous. For $f: \IR_+ \to X$ define
\begin{align*}
L(f):=\limsup_{t\downarrow 0}\frac{1}{t}\|f(t+\cdot)-f\|_{\infty} \in [0,\infty].
\end{align*}

On $\Ce_{\textrm{ub}}(\IR_+;X)$ we consider the derivation operator $Af:= -f'$ with maximal domain $\Ce_{\textrm{ub}}^1(\IR_+;X)$. Then $-A$ generates the strongly continuous left translation semigroup $T_l= (T_l(t))_{t \geq 0}$, where $(T_l(t)f)(\cdot) = f(t + \, \cdot)$. By definition, the associated Favard spaces are given by
\begin{align*}
 \mathrm{F}_k = \{f \in \Ce_{\textrm{ub}}^{k-1}(\IR_+;X) \mid L(f^{(k-1)}) < \infty \} = \Ce_{\textrm{ub}}^{k-1,1}(\IR_{+};X) \quad \quad (k\in \IN).
\end{align*}
Moreover, $T_l$ is of type $(1,0)$ and if $f\in \Ce_{\textrm{ub}}(\IR_+;X)$ and $\lambda \in \IC_+$, then
\begin{align*}
 ((\lambda + A)^{-1}f)(0) = \int_0^\infty \e^{-\lambda t} (T_l(t)f)(0) \, \d t = \int_0^\infty \e^{-\lambda t} f(t) \, \d t = \hat{f}(\lambda),
\end{align*}
with $\hat{f}: \IC_+ \to X$ the (vector-valued) Laplace transform of $f$. This identity enables us to convert our approximation result for uniformly bounded semigroups into an inversion formula for the Laplace transform. To this end, let
\begin{align*}
 r_n(z) = \frac{b_{n,1}}{\lambda_{n,1}-z} + \dots + \frac{b_{n,n+1}}{\lambda_{n,n+1}-z} \quad \quad (z \in \IC \setminus \{\lambda_{n,1},\dots,\lambda_{n,n+1} \})
\end{align*}
be the partial fraction decomposition of the $n$-th subdiagonal Pad\'{e} approximation. Applying Theorem~\ref{main result uniformly bounded} to $T_{l}$ and evaluating at zero yields the subsequent result for all $f\in \Ce_{\textrm{ub}}^k(\IR_+;X)$. Lemma~\ref{Kovacs-lemma} then extends it to all $f\in \Ce_{\textrm{ub}}^{k-1,1}(\IR_+;X)$.

\begin{corollary}
\label{rational laplace inversion}
Let $X$ be a Banach space and $f \in \Ce_{\textrm{ub}}^{k-1,1}(\IR_+;X)$ for some $k \in \IN$. Then for all $t>0$ and all $n \in \IN$ such that $n \geq k - \frac{1}{2}$ the estimate
\begin{align*}
 \left\|\sum_{j=1}^{n+1} \frac{b_{n,j}}{t}\hat{f}\left(\frac{\lambda_{n,j}}{t} \right) - f(t) \right \|_X
 \leq C(k) t^{k} (n+1)^{-k + \frac{1}{2}} L(f^{(k-1)})
\end{align*}
holds true with $C(k)$ given by \eqref{C(alpha)}. In particular, $\sum_{j=1}^{n+1} \frac{b_{n,j}}{t}\hat{f}\left(\frac{\lambda_{n,j}}{t} \right)$ converges to $f(t)$ with rate $\mathcal{O}(n^{-k + \frac{1}{2}})$, locally uniformly in $t$.
\end{corollary}

\begin{remark}
The Laplace inversion formula from Corollary~\ref{rational laplace inversion} actually converges for any $f\in \Ce_{\textrm{ub}}(\IR_{+};X)$ that is $\alpha$-H\"{o}lder continuous for some $\alpha \in (\frac{1}{2},1)$ with rate depending on $\alpha$. This follows again from Theorem~\ref{main result uniformly bounded}, using that such an $f$ is contained in the real interpolation space $(\Ce_{\textrm{ub}}(\IR_+;X), D(A))_{\alpha,\infty}$, which continuously embeds into $\D(A^a)$ for any $a<\alpha$, see Propositions~6.6.3 and B.2.6 in \cite{Haase}.
\end{remark}

\begin{remark}
We emphasize that Corollary~\ref{rational laplace inversion} provides a Laplace inversion formula that does not require any knowledge of derivatives of $\hat{f}$ and only uses finite sums as approximants, compare with e.g.\@ \cite{Jara}, \cite{ABHN}. Moreover, $C(k)$ can be computed explicitly, see also Remark~\ref{Special values of C(alpha)}.
\end{remark}

\paragraph{\textbf{Acknowledgments}}
The authors want to thank Hans Zwart for first coming up with the idea to use the results in \cite{Haase-Rozendaal} to prove something along the lines of Theorem~\ref{main result gamma-stable}, and Markus Haase for providing numerous useful comments and suggestions. The first author wants to thank Frank Neubrander for his kind hospitality and the inspiring discussions during his visit at Louisiana State University. The first author is supported by ``Studienstiftung des deutschen Volkes". The second author is supported by NWO-grant 613.000.908 ``Applications of Transference Principles".


\begin{bibdiv}
\begin{biblist}

\bib{ABHN}{book}{
      author={\sc{W. Arendt}},
      author={\sc{C.J.K. Batty}},
      author={\sc{M. Hieber}},
      author={\sc{F. Neubrander}},
       title={Vector-valued {L}aplace {T}ransforms and {C}auchy {P}roblems},
      series={Monographs in Mathematics},
   publisher={Birkh\"auser/Springer Basel AG, Basel},
        date={2011},
      volume={96},
        ISBN={978-3-0348-0086-0},
}

\bib{Beckenbach-Bellman}{book}{
      author={\sc{E.F. Beckenbach}},
      author={\sc{R. Bellman}},
       title={Inequalities},
      series={Ergebnisse der Mathematik und ihrer Grenz-gebiete},
   publisher={Springer-Verlag},
     address={Berlin},
        date={1961},
      volume={30},
}

\bib{Brenner-Thomee}{article}{
      author={\sc{P. Brenner}},
      author={\sc{V. Thom{\'e}e}},
       title={On rational approximations of semigroups},
        date={1979},
     journal={SIAM J. Numer. Anal.},
      volume={16},
      number={4},
       pages={683\ndash 694},
}

\bib{Cowling-Doust-McIntosh-Yagi}{article}{
      author={\sc{M. Cowling}},
      author={\sc{I. Doust}},
      author={\sc{A. McIntosh}},
      author={\sc{A. Yagi}},
       title={Banach space operators with a bounded {$H^\infty$} functional
  calculus},
        date={1996},
        ISSN={0263-6115},
     journal={J. Austral. Math. Soc. Ser. A},
      volume={60},
      number={1},
       pages={51\ndash 89},
}

\bib{Ehle}{article}{
      author={\sc{B.L. Ehle}},
       title={{$A$}-stable methods and {P}ad\'e approximations to the
  exponential},
        date={1973},
        ISSN={0036-1410},
     journal={SIAM J. Math. Anal.},
      volume={4},
       pages={671\ndash 680},
}

\bib{Haase}{book}{
      author={\sc{M. Haase}},
       title={The {F}unctional {C}alculus for {S}ectorial {O}perators},
      series={Operator Theory: Advances and Applications},
   publisher={Birkh{\"a}user Verlag},
     address={Basel},
        date={2006},
      volume={169},
}

\bib{Haase3}{article}{
      author={\sc{M. Haase}},
       title={Transference principles for semigroups and a theorem of
  {P}eller},
        date={2011},
     journal={J. Funct. Anal.},
      volume={261},
      number={10},
       pages={2959\ndash 2998},
}

\bib{Haase-Rozendaal}{article}{
      author={\sc{M. Haase}},
      author={\sc{J. Rozendaal}},
       title={Functional calculus for semigroup generators via transference},
      eprint={http://arxiv.org/abs/1301.4934}
}

\bib{hairer2004solving}{book}{
      author={\sc{E. Hairer}},
      author={\sc{G. Wanner}},
       title={Solving {O}rdinary {D}ifferential {E}quations {II}: {S}tiff and
  {D}ifferential-{A}lgebraic {P}roblems},
      series={Springer Series in Computational Mathematics},
   publisher={Springer},
     address={Berlin, {H}eidelberg},
        date={2004},
      volume={14},
}


\bib{Hille-Phillips}{book}{
      author={\sc{E. Hille}},
      author={\sc{R.S. Phillips}},
       title={Functional {A}nalysis and {S}emigroups},
      series={American Mathematical Society Colloquium Publications},
   publisher={American Mathematical Society},
     address={Providence, R. I.},
        date={1957},
      volume={31},
}

\bib{Jara}{article}{
      author={\sc{P. Jara}},
      author={\sc{F. Neubrander}},
      author={\sc{K. {\"O}zer}},
       title={Rational inversion of the {L}aplace transform},
        date={2012},
     journal={J. Evol. Equ.},
      volume={12},
      number={2},
       pages={435\ndash 457},
}

\bib{Kalton-Weis}{unpublished}{
    	author={\sc{N. Kalton}},
    	author={\sc{L. Weis}},
     	 title={The $\textrm{H}^{\infty}$-functional calculus and square function estimates},
	  note={Unpublished manuscript},
          year={2004},
}


\bib{Kovacs}{article}{
      author={\sc{M. Kov\'{a}cs}},
       title={A remark on the norm of integer order Favard spaces},
        date={2005},
     journal={Semigroup Forum},
      volume={71},
      number={3},
       pages={462\ndash 470},
}

\bib{Kovacs-Diss}{thesis}{
      author={\sc{M. Kov\'{a}cs}},
       title={On {Q}ualitative {P}roperties and {C}onvergence of
  {T}ime-discretization {M}ethods for {S}emigroups},
        type={Ph.D. Thesis},
      school={Louisiana State University},
        date={2004},
      eprint={http://etd.lsu.edu/docs/available/etd-07082004-143318/}
}

\bib{Larsson-Thomee}{article}{
    author = {\sc{S. Larsson}},
    author = {\sc{V. Thom{\'e}e}},
    author = {\sc{L.B. Wahlbin}},
     title = {Finite-element methods for a strongly damped wave equation},
   journal = {IMA J. Numer. Anal.},
    volume = {11},
      year = {1991},
    number = {1},
     pages = {115\ndash142},
}
	
\bib{NOS}{article}{
      author={\sc{F. Neubrander}},
      author={\sc{K. \"Ozer}},
      author={\sc{T. Sandmaier}},
       title={Rational approximation of semigroups without scaling and
  squaring},
     journal={Discrete Contin. Dyn. Syst.},
      volume={33},
        year={2013},
      number={11\&12},
       pages={5305\ndash 5317}
}

\bib{Frank-Lee}{article}{
      author={\sc{F. Neubrander}},
      author={\sc{L. Windsperger}},
       title={Sharp growth estimates for subdiagonal rational {P}ad\'{e}
  approximations (first draft)},
      eprint={https://www.math.lsu.edu/~neubrand/TheComputerEstimates2013.pdf},
}

\bib{perron1913lehre}{book}{
      author={\sc{O. Perron}},
       title={Die {L}ehre von den {K}ettenbr{\"u}chen},
   publisher={B.G. Teubner},
     address={Leipzig, {Berlin}},
        date={1913},
      volume={1-2},
}


\bib{RudinRC}{book}{
      author={\sc{W. Rudin}},
       title={Real and {C}omplex {A}nalysis},
     edition={3},
   publisher={McGraw-Hill Book Co.},
     address={New York},
        date={1987},
}

\bib{Sz.-Nagy}{article}{
    	author={\sc{B. de Sz.~Nagy}},
     	title={On uniformly bounded linear transformations in {H}ilbert
              space},
   		journal={Acta Univ. Szeged. Sect. Sci. Math.},
    	volume={11},
      year={1947},
     	pages={152\ndash157},
}

\bib{vanNeerven}{incollection}{
      author={\sc{J. van Neerven}},
       title={{$\gamma$}-radonifying operators -- a survey},
   booktitle={The {AMSI}-{ANU} {W}orkshop on {S}pectral {T}heory and {H}armonic {A}nalysis},
      series={Proc. Centre Math. Appl. Austral. Nat. Univ.},
      volume={44},
       pages={1\ndash61},
   publisher={Austral. Nat. Univ.},
     address={Canberra},
        year={2010},
}

\bib{Weis}{article}{
      author={\sc{L. Weis}},
       title={Operator-valued {F}ourier multiplier theorems and maximal {$L_p$}-regularity},
     journal={Math. Ann.},
      volume={319},
        year={2001},
      number={4},
       pages={735\ndash758},
}

\bib{Widder}{book}{
      author={\sc{D.V. Widder}},
       title={The {L}aplace {T}ransform},
      series={Princeton Mathematical Series},
   publisher={Princeton University Press},
     address={Princeton, N. J.},
        date={1941},
      volume={6},
}

\bib{Lee-Diss}{thesis}{
      author={\sc{L. Windsperger}},
       title={Operational {M}ethods for {E}volution {E}quations},
        type={Ph.D. Thesis},
      school={Louisiana State University},
        date={2012},
      eprint={http://etd.lsu.edu/docs/available/etd-07112012-204148/}
}


	



\end{biblist}
\end{bibdiv}
\end{document}